\documentclass[amstex,12pt]{amsart}
\usepackage{mathrsfs}
\usepackage{amsfonts}
\usepackage{amsfonts, amsmath, amsthm, amssymb}
\usepackage{array}
\usepackage{latexsym, euscript, epic, eepic}
\usepackage{tikz}
\usetikzlibrary{matrix}

\newtheorem{theorem}{Theorem}
\newtheorem{lemma}[theorem]{Lemma}
\newtheorem{proposition}[theorem]{Proposition}
\newtheorem{corollary}[theorem]{Corollary}

\theoremstyle{definition}

\newtheorem*{ackn}{Acknowledgments}
\newtheorem*{definition}{Definition}

\newtheorem*{remark}{Remark}

\begin{document}

\title[Strongly symmetric homeomorphisms]{Strongly symmetric homeomorphisms on the real line with uniform continuity} 

\author[H. Wei]{Huaying Wei} 
\address{Department of Mathematics and Statistics, Jiangsu Normal University \endgraf Xuzhou 221116, PR China} 
\email{hywei@jsnu.edu.cn} 

\author[K. Matsuzaki]{Katsuhiko Matsuzaki}
\address{Department of Mathematics, School of Education, Waseda University \endgraf
Shinjuku, Tokyo 169-8050, Japan}
\email{matsuzak@waseda.jp}

\makeatletter
\@namedef{subjclassname@2020}{%
\textup{2020} Mathematics Subject Classification}
\makeatother
\subjclass[2020]{Primary 30C62, 42A45; Secondary 30H35, 26A46, 37E10}
\keywords{quasiconformal, asymptotically conformal, strongly quasisymmetric, symmetric homeomorphism, BMO function, VMO function, Carleson measure, $A_\infty$-weight, barycentric extension, Beurling--Ahlfors extension}
\thanks{Research supported by 
Japan Society for the Promotion of Science (KAKENHI 18H01125 and 21F20027).}

\begin{abstract}
We investigate strongly symmetric homeo\-morphisms of the real line
which appear in harmonic analysis aspects of quasiconformal Teichm\"uller theory.
An element in this class can be characterized by
a property that it can be extended quasiconformally to the upper half-plane
so that its complex dilatation induces a vanishing Carleson measure.
However, differently from the case on the unit circle, strongly symmetric homeomorphisms on the real line
are not preserved under either the composition or the inversion.
In this paper, we present the difference and the relation between these two cases.
In particular, 
we show that if uniform continuity is assumed for strongly symmetric homeo\-morphisms of the real line, then
they are preserved by those operations. We also show that the barycentric extension of  
uniformly continuous one induces a vanishing Carleson measure and so do the composition and the inverse of those 
quasiconformal homeomorphisms of the upper half-plane.
\end{abstract}
 
\maketitle

\section{Introduction and statement of the main results}

The universal Teichm\"uller space and its subspaces are regarded as the spaces consisting of quasiconformal mappings 
on the complex plane. By introducing various particular properties to these mappings
from view points of complex analysis and harmonic analysis,
we can study those concepts through such subspaces reflecting their properties.
For instance, studies on Teichm\"uller spaces of integrable complex dilatations with Weil--Petersson metrics are in 
Cui \cite{Cu}, Takhtajan and Teo \cite{TT}, and Shen \cite{Sh19}, 
those of BMO and VMO functions are in Astala and Zinsmeister \cite{AZ} and Shen and Wei \cite{SW}, and those of
$C^{1+\alpha}$-diffeomorphisms are in \cite{Mat}.

In the conformally invariant formulation, Teichm\"uller spaces defined on the upper half-plane $\mathbb U$ are  
the same as those defined on the unit disk $\mathbb D$.
However, if we consider subspaces of the universal Teichm\"uller space 
by imposing certain conditions on quasiconformal and quasisymmetric mappings,
the theory can differ greatly depending on whether 
the conditions are placed on the compact set (the unit circle $\mathbb S$) or on the non-compact set (the real line $\mathbb R$). 

In this paper, we study the class of strongly symmetric homeomorphisms in the non-compact setting. 
A sense-preserving homeomorphism $h$ of $\mathbb R$ is called strongly symmetric if
$h$ is locally absolutely continuous, 
$h'$ is an $A_\infty$-weight,  and $\log h'$ is a VMO function.
We denote the set of all strongly symmetric homeomorphisms on $\mathbb R$ by
${\rm SS}(\mathbb R)$. The set ${\rm SS}(\mathbb S)$ of those on $\mathbb S$, which defines the original VMO Teichm\"uller space as in \cite{SW}, 
is a natural counterpart to $\rm SS(\mathbb{R})$.

The situation on $\mathbb R$ is more complicated than that on $\mathbb S$ because one has 
to worry about behavior at $\infty$. 
Typically, we have found a phenomenon that $\rm SS(\mathbb{R})$ does not constitute a group
by the composition of mappings in \cite{WM19} whereas
$\rm SS(\mathbb{S})$ is a group.
This causes a trouble in the theory of Teichm\"uller spaces. The VMO Teichm\"uller space $T_v(\mathbb R)$
is defined as the set of all equivalence classes of ${\rm SS}(\mathbb R)$ by affine transformations, 
and its complex analytic structure is studied in
\cite{Sh19} and \cite{WM20}. On the contrary, homogeneity of Teichm\"uller space is important to consider the group of
automorphisms of the Teichm\"uller space and also to introduce an invariant metric with respect to the analytic structure;
$T_v(\mathbb R)$ lacks this nature.

In this paper, we consider a condition under which ${\rm SS}(\mathbb R)$ is preserved by the composition, and
prove the following result.

\begin{theorem}\label{thm2}
If $g, h \in {\rm SS}(\mathbb{R})$ and $h^{-1}$ is uniformly continuous on $\mathbb{R}$, 
then $g \circ h^{-1} \in {\rm SS}(\mathbb{R})$. 
\end{theorem}

Let ${\rm SS}_{\rm uc}(\mathbb R)$
denote a subset of ${\rm SS}(\mathbb R)$ consisting of all elements $h$ such that both $h$ and $h^{-1}$ 
are uniformly continuous.
Then, ${\rm SS}_{\rm uc}(\mathbb R)$ becomes a group by Theorem \ref{thm2}. 
Every element $h \in {\rm SS}_{\rm uc}(\mathbb R)$ acts on $T_v(\mathbb R)$ as an automorphism
that maps the equivalence class of $h$ to the origin of $T_v(\mathbb R)$ (see Section 3). 
Hence, ${\rm SS}_{\rm uc}(\mathbb R)$ is
embedded into the group ${\rm Aut}(T_v(\mathbb R))$ of biholomorphic automorphisms of $T_v(\mathbb R)$, 
which plays the role of 
the Teichm\"uller modular group of $T_v(\mathbb R)$.
In order to see that the action of $h$ is biholomorphic, the following investigation on
quasiconformal extension is necessary.

The complex dilatation $\mu_F$ of a quasiconformal homeomorphism $F$ is defined by
$\mu_F = F_{\bar z}/F_z$. It satisfies $\Vert \mu_F \Vert_\infty<1$.
Let $\mathcal M(\mathbb U)$ be the set of all measurable functions $\mu$ on $\mathbb U$ such that
$\Vert \mu \Vert_\infty<1$ and $|\mu(z)|^2dxdy/y$ is a Carleson measure on $\mathbb U$. In addition, if $|\mu(z)|^2dxdy/y$ is a 
vanishing Carleson measure, then the subset of all such $\mu$ is denoted by $\mathcal M_0(\mathbb U)$.

The following chain rule of complex dilatations is obtained by refinement of the argument in Cui and Zinsmeister 
\cite[Lemma 10]{CZ} who showed the first statement in the case that $G$ is the identity map.

\begin{theorem}\label{biLipschitz}
Let $G$ and $H$ be quasiconformal homeomorphisms of $\mathbb U$ onto itself, and
assume that $H$ is bi-Lipschitz with respect to the hyperbolic metric on $\mathbb U$. 
Then, $(1)$ $\mu_{G \circ H^{-1}}$ belongs to $\mathcal M(\mathbb U)$ if $\mu_G,\, \mu_H \in \mathcal M(\mathbb U)$; 
$(2)$ $\mu_{G \circ H^{-1}}$ belongs to $\mathcal M_0(\mathbb U)$ if $\mu_G,\, \mu_H \in \mathcal M_0(\mathbb U)$ 
and in addition if the boundary extension $h^{-1}$ of $H^{-1}$ to $\mathbb R$ is uniformly continuous.
\end{theorem}
 
We remark that if the uniform continuity is dropped then statement (2) is no longer valid due to the lack of group structure of $\rm SS(\mathbb R)$. The relation between $\mathcal M_0(\mathbb U)$ and $\rm SS(\mathbb R)$ will be given in Proposition \ref{SSext}. 

By statement (2) of Theorem \ref{biLipschitz}, we can define a biholomorphic automorphism of $\mathcal M_0(\mathbb U)$
induced by some quasiconformal extension of $h \in {\rm SS}_{\rm uc}(\mathbb R)$. 
Suppose that $h$ extends to a bi-Lipschitz quasiconformal
homeomorphism $H$ of $\mathbb U$ such that $\mu_H \in \mathcal M_0(\mathbb U)$, which is known to be always the case
independently of Theorem \ref{DEonR} below (see Proposition \ref{step}).
Then, by representing any element of $\mathcal M_0(\mathbb U)$ by $\mu_G$
for a quasiconformal homeomorphism $G$ of $\mathbb U$, we have the right translation $r_H:\mathcal M_0(\mathbb U) \to \mathcal M_0(\mathbb U)$
by the correspondence $\mu_G \mapsto \mu_{G \circ H^{-1}}$. Standard arguments show that $r_H$ is biholomorphic, and
moreover, this action is projected down to $T_v(\mathbb R)$ under the Teichm\"uller projection to induce a
biholomorphic automorphism $R_h:T_v(\mathbb R) \to T_v(\mathbb R)$ well defined by $h \in {\rm SS}_{\rm uc}(\mathbb R)$
(Theorem \ref{action}).

There are several ways to extend quasisymmetric homeomorphisms of $\mathbb S$ and $\mathbb R$
to quasiconformal homeomorphisms. The classical one is due to Beurling and Ahlfors \cite{BA}, and
its variants and modified versions are also introduced by Semmes \cite{Se} and by Fefferman, Kenig and Pipher \cite{FKP}.
Including the barycentric extension introduced by Douady and Earle \cite{DE}, all of them have a property that
the extension map is a bi-Lipschitz diffeomorphism with respect to the hyperbolic metric.
In addition, the conformal naturality of the barycentric extension is useful in the theory of quasiconformal mappings,
in particular, when we consider a M\"obius group action on the Teichm\"uller space.
However, it is so far unknown whether the complex dilatation of the barycentric extension $e(h)$ of a strongly symmetric homeomorphism $h$ 
of $\mathbb R$ induces a vanishing Carleson measure on $\mathbb U$. 
In this paper we prove it does if a strongly symmetric homeomorphism $h$ and its inverse $h^{-1}$ are uniformly continuous on $\mathbb R$,  as stated in the following result.  

\begin{theorem}\label{DEonR}
If $h \in {\rm SS}_{\rm uc}(\mathbb R)$, then $\mu_{e(h)} \in \mathcal{M}_0 (\mathbb U)$. 
\end{theorem}

This result is a consequence of Theorem \ref{biLipschitz}, and implies that a biholomorphic automorphism $R_h$ of $T_v(\mathbb R)$ for $h \in {\rm SS}_{\rm uc}(\mathbb R)$
is lifted canonically
to the biholomorphic automorphism $r_H$ of $\mathcal{M}_0 (\mathbb U)$ by the barycentric extension $H=e(h)$.

We end this introduction (Section 1) with showing the organization of the rest of this paper. 
(Section 2): Definitions and review of basic results are given. These are concerning strongly symmetric and quasisymmetric homeomorphisms,
BMO and VMO functions, the Muckenhoupt weights, the Carleson measures, 
the spaces $\mathcal M(\mathbb U)$ and $\mathcal M_0(\mathbb U)$ of Beltrami coefficients, and their Teichm\"uller spaces.
(Section 3): Under the assumption of uniform continuity, the group structure is considered.
The proof of Theorem \ref{thm2} is given.
We also prove a similar result in Theorem \ref{thm1} for symmetric homeomorphisms of $\mathbb R$, 
a vanishing class of quasisymmetric homeomorphisms, since it has the property parallel to Theorem \ref{thm2}. 
(Section 4): The composition of quasiconformal homeomorphisms whose complex dilatations satisfy the Carleson measure condition
is considered. The uniform continuity condition is applied in the case for the vanishing Carleson measure condition.
The proof of Theorem \ref{biLipschitz} is given.
(Section 5): The barycentric extension defined on $\mathbb R$ is considered. The proof of Theorem \ref{DEonR} is given.
We also give another proof of Theorem \ref{thm2} based on Theorem \ref{biLipschitz}.
(Section 6): Comparisons between strongly symmetric homeomorphisms on $\mathbb R$ and those on $\mathbb S$ under the conjugation by the Cayley transformation are addressed.

\section{Preliminaries}

\subsection{Quasisymmetric homeomorphisms}
As background knowledge, the definitions of quasisymmetric homeomorphisms and the universal Teichm\"uller space
are given. Including the concept of quasiconformal mapping, a basic reference of those is \cite{Ah}.

\begin{definition}
An increasing homeomorphism $h$ of the real line $\mathbb{R}$ onto itself is said to be 
{\it quasisymmetric} 
if there exists a constant $M \geq 1$ such that 
$$
\frac{1}{M} \leq \frac{h(x+t)-h(x)}{h(x)-h(x-t)}\leq M
$$
for all $x\in\mathbb{R}$ and $t>0$. 
The least possible value of such $M$ is called the quasisymmetry constant of $h$.
\end{definition}

If we define a measure $m_h$ by $m_h(E)=|h(E)|$ for a measurable subset $E \subset \mathbb R$
with respect to the Lebesgue measure $|\cdot|$, 
then the boundedness of the quasisymmetry quotient of $h$ is equivalent to that $m_h$ is a {\it doubling measure}, i.e.,
there is some constant $M' \geq 1$ such that $m_h(2I) \leq M' m_h(I)$ 
for every bounded closed interval $I=[x-t,x+t]$ and its double
$2I=[x-2t,x+2t]$.

Beurling and Ahlfors \cite{BA} proved the following theorem.

\begin{proposition}\label{BAext}
An increasing homeomorphism $h$ of the real line $\mathbb{R}$ onto itself is
quasisymmetric if and only if there exists some quasiconformal homeomorphism of the upper half-plane $\mathbb{U}$ onto itself that is continuously extendable to the boundary map $h$. 
\end{proposition}

This quasiconformal extension is explicitly written in terms of  $h$, and is called the Beurling--Ahlfors extension in the literature.  
Later, Douady and Earle \cite{DE} gave a quasiconformal extension of a quasisymmetric homeomorphism, 
called the barycentric extension, in a conformally natural way. We will explain this extension in Section 4.

\begin{definition}
Let $\rm QS(\mathbb{R})$ denote the group of all quasisymmetric homeomorphisms of $\mathbb{R}$. 
The {\it universal Teichm\"uller space} $T$ is  defined as the group $\rm QS(\mathbb{R})$ modulo the left action of the group $\rm Aff(\mathbb R)$ of all real affine mappings $z \mapsto az + b, \, a > 0, b \in \mathbb R$, i.e., 
$T = {\rm Aff}(\mathbb R)\backslash {\rm QS}(\mathbb{R})$. 
\end{definition}

See monographs \cite{Le, Na} for comprehensive introduction on Teichm\"uller spaces.

Let $M(\mathbb{U})$ denote the open unit ball of the Banach space $L^{\infty}(\mathbb{U})$ of essentially bounded measurable functions  
on $\mathbb{U}$. An element in $M(\mathbb{U})$ is called a {\it Beltrami coefficient}. 
By the measurable Riemann mapping theorem, a Beltrami coefficient $\mu \in M(\mathbb U)$ determines uniquely
a quasiconformal homeomorphism $F$ on $\mathbb U$ with its complex dilatation $\mu_F=F_{\bar z}/F_{z}$ equal to $\mu$
up to post composition with conformal mappings. See \cite{Ah}.

By Proposition \ref{BAext} with the measurable Riemann mapping theorem, the universal Teich\-m\"ul\-ler space $T$ can be 
also defined as the set of all equivalence classes $[\mu]$ of $\mu \in M(\mathbb U)$, where $\mu, \mu' \in M(\mathbb U)$
are equivalent if they produce quasiconformal homeomorphisms of $\mathbb U$ onto itself 
having the same boundary extension to $\mathbb R$. We call the quotient map $\pi:M(\mathbb U) \to T$
the {\it Teichm\"uller projection}.

\subsection{BMO, VMO, and $A_\infty$-weights}

The functions in $\rm BMO(\mathbb{R})$ are characterized by the boundedness of their mean oscillations over intervals. The functions in  $\rm VMO(\mathbb{R})$ are those with additional property that their mean oscillations over small intervals are small. To be precise: 

\begin{definition}
A locally integrable  function $u$ on $\mathbb R$ belongs to  BMO  if
$$
\Vert u \Vert_{ \rm BMO(\mathbb{R})} = \sup_{I \subset \mathbb R}\frac{1}{|I|} \int_I |u(x)-u_I| dx <\infty,
$$
where the supremum is taken over all bounded intervals $I$ on $\mathbb R$ and $u_I$ denotes the integral mean of $u$
over $I$. The set of all BMO functions on $\mathbb R$ is denoted by ${\rm BMO}(\mathbb R)$.
This is regarded as a Banach space with the BMO-norm $\Vert \cdot \Vert_{ \rm BMO(\mathbb{R})}$
by ignoring the difference of constant functions. 
It is said that $u \in {\rm BMO}(\mathbb R)$ belongs to VMO if
$$ 
\lim_{|I| \to 0}\frac{1}{|I|} \int_I |u(x)-u_I| dx=0,
$$
and the set of all such functions is denoted by ${\rm VMO}(\mathbb R)$.
\end{definition}

The spaces $\rm BMO(\mathbb S)$ and $\rm VMO(\mathbb S)$ can be defined in the same way. 

Sarason \cite{Sa} proved that 
$\rm VMO(\mathbb{R})$ is the closure of ${\rm BMO}(\mathbb{R}) \cap {\rm UC}(\mathbb R)$ in the BMO-norm,
where ${\rm UC}(\mathbb R)$ denotes the set of uniformly continuous functions on $\mathbb R$. 
In particular, $\rm VMO(\mathbb{R})$ is a closed subspace of ${\rm BMO}(\mathbb{R})$.

\begin{remark}
{\rm
The space ${\rm VMO}(\mathbb S)$ is the closure of $C^\infty(\mathbb S)$
in the BMO-norm.
Differently from the case of $\mathbb S$, however, it was shown by Martell, Mitrea et al. \cite[Theorem 1.8, Corollary 1.7]{MMMM} that
$L^\infty(\mathbb R) \cap {\rm UC}(\mathbb R)$ is not dense in ${\rm VMO}(\mathbb R)$, while ${\rm BMO}(\mathbb{R}) \cap C^\alpha(\mathbb R)$
is dense in $\rm VMO(\mathbb{R})$ in the BMO-norm. Here, 
$C^\alpha(\mathbb R)$ denotes the set of
H\"older continuous functions on $\mathbb R$ of order $\alpha \in (0,1)$.
}
\end{remark}

BMO functions satisfy the following {\it John--Nirenberg inequality}
(see \cite[Section VI.2]{Ga}, \cite[Section IV.1.3]{St2}).

\begin{proposition}\label{JN}
There exists two universal positive constants  $C_1, C_2>0$ such that for any BMO function $u$ on $\mathbb R$ $($or on $\mathbb S$$)$,  
any bounded closed interval $J \subset I$ on
any interval $I \subset \mathbb R$ $($or $I \subset \mathbb S$$)$,
the inequality
\begin{equation*}
    \frac{1}{|J|}|\{z \in J: |u(x) - u_J|> t \}|\leq C_1 {\rm exp}\left(\frac{-C_2t}{\Vert u \Vert_{{\rm BMO}(I)}}\right)
\end{equation*}
holds for all $t > 0$, where $\Vert u \Vert_{{\rm BMO}(I)}$ is the BMO-norm of $u$ on $I$.  
\end{proposition}

The exponentials of BMO functions are closely related to the Muckenhoupt weights.
There are several equivalent definitions of $A_{\infty}$-weights (see \cite{CF}), and the following is one of them.

\begin{definition} 
A locally integrable non-negative measurable function $\omega \geq 0$ on $\mathbb R$ is called a {\it weight}. We say that $\omega$ is an {\it $A_{\infty}$-weight} 
if there exist two positive constants $C$ and $\alpha$ such that 
$$
\frac{\int_E\omega(x)dx}{\int_I\omega(x)dx}\leq C\left(\frac{|E|}{|I|}\right)^{\alpha}
$$
whenever $I\subset \mathbb{R}$ is a bounded closed interval and $E\subset I$ a measurable subset.  
\end{definition}

A weight $\omega$ is called {\it doubling} if the measure $\omega(x)dx$ is doubling.
If we use the above definition, it is easy to see that 
an $A_{\infty}$-weight is doubling. But, the converse is not.
Fefferman and Muckenhoupt \cite{FM} provided an example of a weight that satisfies the doubling condition but not $A_{\infty}$. 

If $\omega$ is an $A_{\infty}$-weight, then $\log\omega$ is a BMO function.  
Conversely, if $\log\omega$ is a real-valued BMO function, then $\omega^{\delta}$ is an $A_{\infty}$-weight 
for some small $\delta > 0$. This is a consequence from the John--Nirenberg inequality,
but $\omega$ itself need not be even locally integrable, and thus need not be an $A_{\infty}$-weight (see \cite[p.409]{GR}). 
However,  
if $\log\omega$ is a real-valued VMO function on the unit circle $\mathbb S$, 
then by the John--Nirenberg inequality, $\omega$ is an $A_{\infty}$-weight on $\mathbb S$ (see \cite[p.474]{GR}).
Namely, the local BMO norm of a VMO function $\log\omega$ can be made so small on a small interval that
$\omega$ is a local $A_{\infty}$-weight on this interval. Then, 
this is in fact an $A_{\infty}$-weight on $\mathbb S$ by its compactness.

\subsection{Beltrami coefficients inducing Carleson measures}
We define the spaces of Beltrami coefficients characterizing particular quasiconformal homeomorphisms
considered in our research.

\begin{definition}
Let $\lambda$ be a positive Borel measure on the upper half-plane $\mathbb{U}$. 
We say that $\lambda$ is a {\em Carleson measure} if
$$
\Vert \lambda \Vert_c  = \sup_{I \subset \mathbb R} \frac{\lambda(I \times (0,|I|])}{|I|} < \infty,
$$
where the supremum is taken over all bounded closed interval $I \subset \mathbb R$ and
$I \times (0,|I|] \subset \mathbb U$ is a Carleson box. The set of all Carleson measures on $\mathbb U$ is denoted by
${\rm CM}(\mathbb U)$. 
A Carleson measure $\lambda \in {\rm CM}(\mathbb U)$ is called {\em vanishing} if
$$\lim_{|I| \to 0}\frac{\lambda(I \times (0,|I|])}{|I|} = 0.$$
The set of all vanishing Carleson measures on $\mathbb U$ is denoted by
${\rm CM}_0(\mathbb U)$.
\end{definition}

\begin{definition}
Let $\mathcal{L}(\mathbb{U})$ be the Banach space of all essentially bounded measurable functions $\mu$ on $\mathbb{U}$ 
such that $\lambda_{\mu} \in {\rm CM}(\mathbb{U})$ 
for $d\lambda_{\mu}(z) = |\mu(z)|^2\rho_{\mathbb{U}}(z)dxdy$. Here, $\rho_{\mathbb U}$ is the hyperbolic density on $\mathbb U$. 
The norm of $\mathcal{L}(\mathbb{U})$ is given by $\Vert \mu \Vert_\infty+\Vert \lambda_\mu \Vert_c^{1/2}$.
Let $\mathcal{L}_0(\mathbb{U})$ be the subspace of $\mathcal{L}(\mathbb{U})$ consisting of all elements $\mu$ 
such that $\lambda_{\mu} \in {\rm CM}_0(\mathbb{U})$.  Moreover, we set the corresponding spaces of Beltrami coefficients as  $\mathcal{M}(\mathbb{U}) = M(\mathbb{U}) \cap \mathcal{L}(\mathbb{U})$, and $\mathcal{M}_0(\mathbb{U}) = M(\mathbb{U}) \cap \mathcal{L}_0(\mathbb{U})$. 
\end{definition}

On the unit disk $\mathbb D$, the corresponding spaces of Carleson measures
$\rm CM(\mathbb{D})$ and vanishing Carleson measures $\rm CM_0(\mathbb{D})$ are defined in the same way (see \cite[p.231]{Ga}),
and so are $\mathcal{M}(\mathbb{D})$ and $\mathcal{M}_0(\mathbb{D})$.

\subsection{Strongly quasisymmetric homeomorphism}
The notion of strongly quasisymmetric homeomorphisms was introduced by Semmes \cite{Se}. 
This subclass is much related with and also has wide application to some important problems in real and harmonic analysis (see \cite{Da}). 

\begin{definition}
An increasing homeomorphism $h$ of $\mathbb{R}$ onto itself is said to be {\it strongly quasisymmetric} if $h$ is locally absolutely continuous and $h'$ belongs to the class of $A_{\infty}$-weights. Let ${\rm SQS}(\mathbb{R})$ denote the set of all strongly quasisymmetric homeomorphisms of $\mathbb{R}$ onto itself.
\end{definition}

In particular, if $h \in \rm SQS(\mathbb{R})$ then $\log h'$ belongs to $\rm BMO(\mathbb{R})$. 
By the definition of $A_{\infty}$-weight, we see that ${\rm SQS}(\mathbb{R})$ is preserved under the composition of elements.
In addition, by \cite[Lemma 5]{CF}, the inverse operation also preserves ${\rm SQS}(\mathbb{R})$.
Thus, ${\rm SQS}(\mathbb{R})$ is a group.
Moreover,
since an $A_{\infty}$-weight defines a doubling measure,
$\rm SQS(\mathbb{R})$ is a subgroup of $\rm QS(\mathbb{R})$.

\begin{definition}
We say that a strongly quasisymmetric homeomorphism $h \in \rm SQS(\mathbb{R})$ is {\it strongly symmetric} if $\log h'$ belongs to $\rm VMO(\mathbb{R})$. Let $\rm SS(\mathbb{R})$ denote the set of all strongly symmetric homeomorphisms of $\mathbb{R}$. 
\end{definition}

On the unit circle $\mathbb S$, strongly quasisymmetric and symmetric homeomorphisms are defined similarly.
The sets of those are denoted by ${\rm SQS}(\mathbb{S})$ and ${\rm SS}(\mathbb{S})$ respectively.
These classes were investigated in \cite{FHS, SW, We, WS}
during their study of BMO theory on Teichm\"uller spaces. In particular, it is known that
${\rm SS}(\mathbb S)$ is a subgroup of ${\rm SQS}(\mathbb S)$.

Concerning the quasiconformal extension of strongly quasisymmetric homeomorphisms to $\mathbb U$, and conversely
the boundary extension of quasiconformal homeomorphisms with complex dilatations in $\mathcal M(\mathbb U)$,
the results in Fefferman, Kenig and Pipher \cite[Theorems 2.3 and 4.2]{FKP} imply the following claim adapted to our purpose.

\begin{proposition}\label{SQSext}
An increasing homeomorphism $h$ of $\mathbb{R}$ onto itself belongs to ${\rm SQS}(\mathbb R)$ if and only if  
$h$ continuously extends to some quasiconformal homeomorphism of $\mathbb{U}$ onto itself 
whose complex dilatation belongs to $\mathcal M(\mathbb U)$.
\end{proposition}

Precisely, a variant of the Beurling--Ahlfors extension by the heat kernel introduced in \cite{FKP}
gives an appropriate extension (quasiconformal diffeomorphism) for strongly quasisymmetric homeomorphisms, 
while a variant of the Beurling--Ahlfors extension constructed by Semmes \cite{Se} 
is also valid in this case under the assumption that $\log h'$ has small BMO norm. 

Due to the conformal invariance of Carleson measures (see \cite[Lemma VI.3.3]{Ga}), 
the corresponding statement to Proposition \ref{SQSext} for ${\rm SQS}(\mathbb S)$ and $\mathcal M(\mathbb D)$
holds true. However, for strongly symmetric homeomorphisms, the situation is different. 

In \cite[Theorem 4.1]{SW}, the corresponding claim for strongly symmetric homeomorphisms on $\mathbb S$ is
proved though this only asserts the existence of desired quasiconformal extension. Later, under this existence result,
it is shown that the barycentric extension is in fact such an extension (see Remark in Section 5).

\begin{proposition}\label{SW4.1}
A sense-preserving homeomorphism $\varphi$ of $\mathbb S$ onto itself belongs to
$\rm SS(\mathbb S)$ if and only if $\varphi$ continuously extends to some quasiconformal homeomorphism of $\mathbb D$ 
onto itself whose complex dilatation belongs to $\mathcal{M}_0(\mathbb D)$. 
\end{proposition}

This result uses a fact that $C^\infty(\mathbb S)$ is dense in
${\rm VMO}(\mathbb S)$, and
does not directly imply the corresponding statement for ${\rm SS}(\mathbb R)$ and $\mathcal M_0(\mathbb U)$.
Nevertheless, the result itself is satisfied and 
we obtain the following characterization of strongly symmetric homeomorphisms on $\mathbb R$
by their quasiconformal extensions.

\begin{proposition}\label{SSext}
An increasing homeomorphism $h$ of $\mathbb{R}$ onto itself belongs to ${\rm SS}(\mathbb R)$ if and only if  
$h$ continuously extends to some quasiconformal homeomorphism of $\mathbb{U}$ onto itself 
whose complex dilatation belongs to $\mathcal M_0(\mathbb U)$.
\end{proposition}

Indeed, it was proved in \cite[Theorem 4.1]{WM20} that the variant of the Beurling--Ahlfors extension by the heat kernel yields
such a quasiconformal extension. The fact that the above boundary extension is strongly symmetric was obtained 
in \cite[Theorem 2.2]{Sh}.

\begin{definition}
The quotient space $T_b = {\rm Aff}(\mathbb R) \backslash {\rm SQS}(\mathbb{R})$ is called the BMO {\it Teichm\"uller space}.
This can be also defined by $T_b=\pi(\mathcal M(\mathbb U))$.
\end{definition}

This space was introduced by Astala and Zinsmeister \cite{AZ}. By the conformal invariance,
we can also define this by $T_b = \mbox{M\"ob}(\mathbb S) \backslash {\rm SQS}(\mathbb{S})$, where
$\mbox{M\"ob}(\mathbb S)$ is the group of M\"obius transformations keeping $\mathbb S$ invariant.

\begin{definition}
The quotient space $T_v = \mbox{M\"ob}(\mathbb S) \backslash {\rm SS}(\mathbb{S})$ is called the VMO {\it Teichm\"uller space},
and $T_v(\mathbb R)= {\rm Aff}(\mathbb R) \backslash {\rm SS}(\mathbb{R})$ the VMO {\it  Teichm\"uller space on the real line}.
These can be also defined by $T_v=\pi(\mathcal M_0(\mathbb D))$ and $T_v(\mathbb R)=\pi(\mathcal M_0(\mathbb U))$.
\end{definition}

These two VMO Teichm\"uller spaces are different. Considering all of them on $\mathbb R$ by taking the conjugate, 
we have the strict inclusion relations
$T_v \subset T_v(\mathbb R) \subset T_b$ (see Theorem \ref{SR}).
The VMO Teichm\"uller space on the real line was introduced by Shen \cite{Sh}.
The Teichm\"uller spaces $T_b$ and $T_v$ possess the group structure inherited from those of
${\rm SQS}(\mathbb{S})$ and ${\rm SS}(\mathbb{S})$.

\section{Uniform continuity: Proofs of Theorem \ref{thm2} and a theorem for symmetric homeomorphisms}

In this section, we give the real-variable proof of Theorem \ref{thm2} for strongly symmetric homeomorphisms. 
We also show 
that symmetric homeomorphisms defined below have a very similar property to that in Theorem \ref{thm2}.

It is known that each strongly quasisymmetric homeomorphism $h$ induces a bounded linear isomorphism 
$P_h: \rm BMO(\mathbb{R})\to BMO(\mathbb{R})$ by $P_h(u) = u \circ h$ (see Jones \cite{Jo}). We note that the operator 
$P_h$ does not necessarily maps $\rm VMO(\mathbb{R})$ into itself. For example, there exist strongly symmetric homeomorphisms $g$ and $h$ (see Section 6 for specific constructions) such that  
$\log (g\circ h)' = \log g' \circ h + \log h' \notin \rm VMO(\mathbb{R})$. Since $\log h' \in \rm VMO(\mathbb{R})$, 
we see that $P_h(\log g')= \log g' \circ h \notin \rm VMO(\mathbb{R})$. 

However, under the uniform continuity of $h$, the operator $P_h$ maps $\rm VMO(\mathbb{R})$ properly.
Namely, we have the following:

\begin{proposition}\label{pullback}
Let $h \in {\rm SQS}(\mathbb{R})$ such that $h$ is uniformly continuous on $\mathbb{R}$. Then, 
the operator $P_h$ maps 
$\rm VMO(\mathbb{R})$ into $\rm VMO(\mathbb{R})$. 
\end{proposition}

\begin{proof}
We follow the proof of Anderson, Becker and Lesley \cite[Lemma]{ABL}.
To see that this is composed by real analytic arguments, we show their proof.

Let $v=P_h(u)$ for any $u \in {\rm BMO}(\mathbb R)$, 
which is also in $\rm BMO(\mathbb{R})$. For every bounded interval $I \subset \mathbb R$,
we have
$$
\int_I |v(x)-v_I| dx \leq \int_I |v(x)-u_{h(I)}| dx+ \int_I |v_I-u_{h(I)}|dx \leq 2 \int_I |v(x)-u_{h(I)}| dx.
$$
Here, for $E_t=\{y \in h(I): |u(y) - u_{h(I)}|> t \}$, Proposition \ref{JN} implies that
\begin{equation*}
\frac{|E_t|}{|h(I)|}\leq C_1 {\rm exp}\left(\frac{-C_2t}{\Vert u \Vert_{{\rm BMO}(h(I))}}\right).
\end{equation*}
Since the inverse $h^{-1}$ also belongs to ${\rm SQS}(\mathbb R)$, there are positive constants $C$ and $\alpha$ 
for the $A_\infty$-weight $(h^{-1})'$ such that
$$
\frac{|h^{-1}(E_t)|}{|I|} \leq C\left(\frac{|E_t|}{|h(I)|}\right)^\alpha.
$$
We consider the distribution function $\lambda(t)=|h^{-1}(E_t)|$, where
$$
h^{-1}(E_t)=\{x \in I: |u\circ h(x)-u_{h(I)}|> t \}.
$$
Then, we have

\begin{align*}
&\quad \int_I |v(x)-u_{h(I)}| dx=\int_I |u \circ h(x)-u_{h(I)}| dx=\int_0^\infty \lambda(t)dt\\
&\leq CC_1^\alpha|I| \int_0^\infty {\rm exp}\left(\frac{-\alpha C_2 t}{\Vert u \Vert_{{\rm BMO}(h(I))}}\right)dt
=\frac{CC_1^\alpha}{\alpha C_2}|I| \Vert u \Vert_{{\rm BMO}(h(I))}.
\end{align*}
Therefore, we conclude that
$$
\frac{1}{|I|} \int_I |P_h(u)(x)-(P_h(u))_I|dx \leq \frac{2CC_1^\alpha}{\alpha C_2} \Vert u \Vert_{{\rm BMO}(h(I))}.
$$

Since $h$ is uniformly continuous on $\mathbb{R}$, 
we have $|h(I)|\to 0$ as $|I| \to 0$. Hence, the BMO norm $\Vert u \Vert_{{\rm BMO}(h(I))}$
tends uniformly to $0$ as $|I| \to 0$ owning to $u \in \rm VMO(\mathbb{R})$. This implies that
$P_h(u) \in \rm VMO(\mathbb{R})$.
\end{proof}

Theorem \ref{thm2} follows from this proposition.

\begin{proof}[Proof of Theorem \ref{thm2}]
Since $\rm SQS(\mathbb{R})$ is a group, we have $g \circ h^{-1} \in \rm SQS(\mathbb{R})$ 
by the condition $g, h \in \rm SS(\mathbb{R}) \subset \rm SQS(\mathbb{R})$. Noting that 
\begin{equation*}
\log(g\circ h^{-1})' = (\log g' - \log h')\circ h^{-1} = P_{h^{-1}}(\log g' - \log h'),
\end{equation*}
we conclude by Proposition \ref{pullback} that $\log(g\circ h^{-1})' \in \rm VMO(\mathbb{R})$, and thus 
$g\circ h^{-1} \in \rm SS(\mathbb{R})$.
\end{proof}

Theorem \ref{thm2} implies the following as well:

\begin{corollary}\label{3item}
The following statements hold:
\begin{enumerate}
    \item If $h \in \rm SS(\mathbb{R})$ and $h^{-1}$ is uniformly continuous on $\mathbb{R}$, then $h^{-1} \in \rm SS(\mathbb{R})$;
    
    \item If $g, h \in \rm SS(\mathbb{R})$ and $h, h^{-1}$ are uniformly continuous on $\mathbb{R}$, then $g\circ h \in \rm SS(\mathbb{R})$; 
    
    \item The set of elements $h \in \rm SS(\mathbb{R})$ such that both $h$ and $h^{-1}$ are uniformly continuous on $\mathbb{R}$ is a
    subgroup of ${\rm SQS}(\mathbb R)$.
    \end{enumerate}
\end{corollary}

\begin{definition}
The subgroup of ${\rm SQS}(\mathbb R)$ consisting of all elements $h \in \rm SS(\mathbb{R})$ such that both $h$ and $h^{-1}$ are uniformly continuous on $\mathbb{R}$ is denoted by ${\rm SS}_{\rm uc}(\mathbb R)$.
\end{definition}

As mentioned in Section 1, ${\rm SS}_{\rm uc}(\mathbb R)$ acts on the VMO Teichm\"uller space 
$T_v(\mathbb R)={\rm Aff}(\mathbb R) \backslash {\rm SS}(\mathbb{R})$ as a group of its automorphisms.
For any $h \in {\rm SS}_{\rm uc}(\mathbb R)$, this action is defined by
$h_*([g])=[g \circ h^{-1}]$ for every $[g] \in T_v(\mathbb R)$, where $[g]$ denotes the equivalence class
represented by $g \in {\rm SS}(\mathbb{R})$.

Next, we consider a similar problem for symmetric homeomorphisms of $\mathbb R$.

\begin{definition}
A quasisymmetric homeomorphism $h$ of $\mathbb R$ is said to be {\it symmetric} if 
$$ 
\lim_{t\to 0}\,\frac{h(x+t)-h(x)}{h(x)-h(x-t)}=1 
$$
uniformly for all $x\in\mathbb{R}$.
Let $\rm S(\mathbb{R})$ denote the subset
of $\rm QS(\mathbb{R})$ consisting of all symmetric homeomorphisms of $\mathbb{R}$. 
\end{definition}

It is known that $h$ is symmetric if and only if $h$ can be extended to an asymptotically conformal homeomorphism $H$ 
of the upper half-plane $\mathbb{U}$ onto itself (see \cite{Ca, GS}). 
Here, we say that $H$ is {\it asymptotically conformal} if its complex dilatation $\mu_H = H_{\bar z}/H_z$ 
satisfies that
$$
{\rm ess}\!\!\!\!\!\!\!\sup_{0<y < t\qquad} \!\!\!\!\!\!\! |\mu_H(x+iy)| \to 0 \quad (t \to 0).
$$
In fact, the Beurling--Ahlfors extension of $h$ is asymptotically conformal when $h$ is symmetric. 

The class $\rm S(\mathbb{R})$ was first studied by Carleson \cite{Ca} when he discussed absolute continuity of quasisymmetric homeomorphisms. It was investigated in depth later by Gardiner and Sullivan \cite{GS} in their study of the asymptotic Teichm\"uller space $T_0=\mbox{M\"ob}(\mathbb S) \backslash {\rm S}(\mathbb{S})$ by
using ${\rm S}(\mathbb{S})$ similarly defined on the unit circle $\mathbb S$.
Recently, Hu, Wu and Shen \cite{HWS} introduced 
the Teichm\"uller space $T_0(\mathbb R)={\rm Aff}(\mathbb R) \backslash \rm S(\mathbb{R})$ on the real line.
This has been further generalized in \cite{WMp}. 

The inclusion relation $\rm SS(\mathbb{R}) \subset \rm S(\mathbb{R})$ 
is seen from the characterization of VMO functions in Sarason \cite[Theorem 2]{Sa}.
In more detail, by the John--Nirenberg inequality, the local $A_2$-constant for the exponential of a VMO function
tends to $1$ when the interval gets small. This can be applied to show the above inclusion.
See \cite[Lemma 3.3]{Sh}. 

In \cite{WM19}, we constructed counter-examples for showing that the class $\rm S(\mathbb{R})$ does not constitute a group under the composition. To be precise, we have proved that neither the composition nor the inverse preserves this class.  
However, we have the following result similar to Theorem \ref{thm2} for strongly symmetric homeomorphisms. 
This can be regarded as its prototype in the non-compact setting.

\begin{theorem}\label{thm1}
If $g, h \in \rm S(\mathbb{R})$ and $h^{-1}$ is uniformly continuous on $\mathbb{R}$, 
then $g\circ h^{-1} \in \rm S(\mathbb{R})$.
\end{theorem}
\begin{proof}
Suppose that $g\circ h^{-1}$ is not symmetric. Then, for some $\delta>0$,
there are consecutive bounded closed intervals $J_n$ and $J_n'$ in $\mathbb R$ such that
$|J_n|=|J_n'| \to 0$ $(n \to \infty)$ and 
$$
\max\left\{\frac{|g\circ h^{-1}(J_n)|}{|g\circ h^{-1}(J_n')|},\frac{|g\circ h^{-1}(J_n')|}{|g\circ h^{-1}(J_n)|}\right\} \geq 1+\delta.
$$
Without loss of generality,
we may assume that $|g\circ h^{-1}(J_n)| \geq (1+\delta)|g\circ h^{-1}(J_n')|$.
Let $I_n = h^{-1}(J_n)$ and $I_n' = h^{-1}(J_n')$. Since $h^{-1}$ is uniformly continuous, we have $|I_n|\to 0$ and $|I_n'|\to 0$. 
By the symmetry of $g$ with $|g(I_n)| \geq (1+\delta)|g(I_n')|$, we see that
there exists $\varepsilon > 0$ such that 
$|I_n| \geq (1+\varepsilon)|I_n'|$ for all sufficiently large $n$. 
We choose $\widetilde{I}_n \subset I_n$ so that $\widetilde{I}_n$ and $I_n'$ are consecutive intervals in $\mathbb{R}$ 
with $|\widetilde{I}_n|=|I_n'|$. Then, $|I_n| \geq (1+\varepsilon)|\widetilde I_n|$ but
$$
\lim_{n \to \infty} \frac{|h(I_n)|}{|h(\widetilde I_n)|} =\lim_{n \to \infty} \frac{|h(I_n)|}{|h(I'_n)|}=\frac{|J_n|}{|J'_n|}=1.
$$
This contradicts that $h$ is a symmetric homeomorphism.
\end{proof}

As an immediate consequence from Theorem \ref{thm1}, we have: 

\begin{corollary}
The following statements hold:
\begin{enumerate}
    \item If $h \in \rm S(\mathbb{R})$ and $h^{-1}$ is uniformly continuous on $\mathbb{R}$, then $h^{-1} \in \rm S(\mathbb{R})$;
    
    \item If $g, h \in \rm S(\mathbb{R})$ and $h, h^{-1}$ are uniformly continuous on $\mathbb{R}$, then $g\circ h \in \rm S(\mathbb{R})$; 
    
    \item The set of elements $h \in \rm S(\mathbb{R})$ such that both $h$ and $h^{-1}$ are uniformly continuous on $\mathbb{R}$ is a
    subgroup of ${\rm QS}(\mathbb R)$.
    \end{enumerate}
\end{corollary}

Thus, the group of those $h \in {\rm S}(\mathbb{R})$ with both $h$ and $h^{-1}$ being uniformly continuous acts on 
the Teichm\"uller space $T_0(\mathbb R)$.

\section{The chain rule of complex dilatations: Proof of Theorem \ref{biLipschitz}}

The idea of the proof of Theorem \ref{biLipschitz} originally appeared in 
Cui and Zinsmeister \cite[Lemma 10]{CZ} (and in Semmes \cite[lemma 4.8]{Se} for a different statement),
and we supply necessary ingredients for it to make our proof 
more understandable. These include stability of quasi-geodesics, the Carleson embedding theorem, 
and certain properties of $A_\infty$-weights in
the Muckenhoupt theory. 

First, we prepare two lemmas for the proof of this theorem.
For every $z=(x,y) \in \mathbb U$, we define a closed interval $I_z \subset \mathbb R$ by
$I_z=[x-y,x+y]$.
Conversely,
for a closed interval $I=[a,b] \subset \mathbb R$, we define a point $q(I) \in \mathbb U$ by
$q(I)=(\frac{a+b}{2},\frac{b-a}{2})$.
Let $\gamma_{a,b}$ denote the hyperbolic geodesic line in the hyperbolic plane $\mathbb U$ joining
two points $a$ and $b$ on the boundary at infinity ${\mathbb R} \cup \{\infty\}$.
Then, 
$q([a,b])$ is the intersection of $\gamma_{a,b}$ and $\gamma_{(a+b)/2,\infty}$.

Hereafter, we use a convenient notation $A \asymp B$, which means that there is a constant $C \geq 1$
satisfying that $C^{-1}A \leq B \leq CA$ uniformly with respect to certain circumstances obvious from the context. 

\begin{lemma}\label{ratio}
Let $F:\mathbb U \to \mathbb U$ be a bi-Lipschitz quasiconformal homeomorphism with respect to the hyperbolic metric
that extends to a quasisymmetric homeomorphism $f:\mathbb R \to \mathbb R$ with $f(\infty)=\infty$.
Then, there is a constant $C \geq 1$ depending only on the bi-Lipschitz constant $L=L(F) \geq 1$ of $F$ such that
$$
C^{-1} \frac{|f(I_z)|}{|I_z|} \leq \frac{{\rm Im}\,F(z)}{{\rm Im}\,z} \leq C \frac{|f(I_z)|}{|I_z|}
$$
for every $z \in \mathbb U$. 
\end{lemma}

\begin{proof}
For $I_z=[a,b]$, the point $z$ is the intersection of the geodesic lines $\gamma_{a,b}$ and $\gamma_{(a+b)/2,\infty}$.
Then, $F(z)$ is the intersection of the quasi-geodesic lines $F(\gamma_{a,b})$ and $F(\gamma_{(a+b)/2,\infty})$.
By the {\it stability of quasi-geodesics} (see \cite[p.401]{BH} and \cite[p.41]{CDP}), 
we see that $F(\gamma_{a,b})$ is within a bounded hyperbolic distance of
$\gamma_{f(a),f(b)}$ and that $F(\gamma_{(a+b)/2,\infty})$ is within a bounded hyperbolic distance of $\gamma_{f((a+b)/2),\infty}$,
where the bounds depend only on the bi-Lipschitz constant $L$. 
This shows that the hyperbolic distance between the intersections $F(z)$ and $\gamma_{f(a),f(b)} \cap \gamma_{f((a+b)/2),\infty}$ is bounded from above by a
constant depending only on $L$. 

On the other hand, by the quasisymmetry of $f$ on $\mathbb{R}$, 
there exists some constant $M \geq 1$ which also depends only on $L$ such that 
$$
M^{-1} \leq \frac{f(b)-f((a+b)/2)}{f((a+b)/2)-f(a)} \leq M.
$$
This shows that the hyperbolic distance between
the intersections $\gamma_{f(a),f(b)} \cap \gamma_{f((a+b)/2),\infty}$ and
$\gamma_{f(a),f(b)} \cap \gamma_{(f(a)+f(b))/2,\infty}=q(f(I_z))$ is bounded from above by a
constant depending only on $L$. 

Combining the boundedness from above of the two hyperbolic distances mentioned above, we conclude that the hyperbolic distance between $F(z)$ and $q(f(I_z))$ is bounded from above by a
constant depending only on $L$. By the formula of the hyperbolic distance $d_{H}$ on the upper half-plane $\mathbb{U}$, for any two points $z, w \in \mathbb{U}$, it holds that
\begin{equation*}
    \sinh\left(\frac{d_H(z, w)}{2}\right) = \frac{|z - w|}{2 ({\rm Im}\,z\, {\rm Im}\,w)^{\frac{1}{2}}}
    \geq \frac{1}{2}\left|\left(\frac{{\rm Im}\,z}{{\rm Im}\,w}\right)^{\frac{1}{2}}-\left(\frac{{\rm Im}\,w}{{\rm Im}\,z}\right)^{\frac{1}{2}}\right|.
\end{equation*}
Thus, 
\begin{equation*}
 C^{-1}{\rm Im}\,F(z) \leq {\rm Im}\, q(f(I_z)) \leq C {\rm Im}\,F(z)
\end{equation*}
for some constant $C \geq 1$ depending only on $L$.
Consequently, we have
$$
C^{-1}\frac{{\rm Im}\, F(z)}{{\rm Im}\,z} \leq \frac{|f(I_z)|}{|I_z|} =\frac{{\rm Im}\, q(f(I_z))}{{\rm Im}\,z} \leq C \frac{{\rm Im}\, F(z)}{{\rm Im}\,z},
$$
which is the required inequalities.
\end{proof}

\begin{lemma}\label{box}
Let $F:\mathbb U \to \mathbb U$ and $f:\mathbb R \to \mathbb R$ be as in Lemma \ref{ratio}.
There is a constant $\alpha \geq 1$ depending only on the bi-Lipschitz constant $L=L(F)$ such that
for every bounded closed interval $I \subset \mathbb R$, the image $F(Q_I)$ of the Carleson box 
$Q_I=I \times (0,|I|] \subset \mathbb U$ is contained in the Carleson box $Q_{\alpha f(I)}$
associated with the interval $\alpha f(I) \subset \mathbb R$ with the same center as $f(I)$ and with length $|\alpha f(I)|=\alpha |f(I)|$.
\end{lemma}

\begin{proof}
We take any point $z$ on the upper side $I \times \{|I|\}$ of the Carleson box $Q_I$ and consider $F(z)$.
Lemma \ref{ratio} shows that
$$
{\rm Im}\, F(z) \asymp \frac{{\rm Im}\, z \cdot|f(I_z)|}{|I_z|}=\frac{1}{2} |f(I_z)|
$$
is satisfied with the comparability constant $C \geq 1$ depending only on $L$.
Since $I \subset I_z \subset 3I$, the quasisymmetry of $f$ with a constant $M \geq 1$ depending only on $L$ implies that 
$|f(I_z)| \leq (2M+1)|f(I)|$. Hence, ${\rm Im}\, F(z)$ is bounded by $C(M+\frac{1}{2})|f(I)|$.

By the stability of quasi-geodesics, the images of the left and the right sides of $Q_I$ under $F$ are
within a bounded hyperbolic distance of the hyperbolic geodesic lines towards $\infty$ from the left and
the right end points of the interval $f(I)$, respectively.
Combining this with the above estimate on the image of the upper side of $Q_I$, we can find the required constant
$\alpha \geq 1$ depending only on $L$ such that $F(Q_I) \subset Q_{\alpha f(I)}$.
\end{proof}

\begin{proof}[Proof of Theorem \ref{biLipschitz}]
Statement (1) for the case of $G={\rm id}$ was proved by Cui and Zinsmeister \cite[Lemma 10]{CZ}. 
We will give the proof of the other statements. 
Our argument also gives a detailed exposition of their proof simultaneously. 
We remark in advance that in virtue of Proposition \ref{SQSext} we can use
the strongly quasisymmetric properties of the boundary extension $h$ of $H$ to $\mathbb R$ in the proof.

Let $I \subset \mathbb{R}$ be a bounded closed interval and $Q_I = I \times (0, |I|]$ the associated Carleson box. Then, 
by change of variables $\zeta=H(z)$, we obtain that
\begin{equation*}
    \begin{split}
        \mathcal{I} & =  \iint_{Q_I} |\mu_{G \circ H^{-1}}(\zeta)|^2 \rho_{\mathbb{U}}(\zeta)d\xi d\eta\\
        & = \iint_{H^{-1}(Q_I)}  \left|\frac{\mu(z) - \nu(z)}{1-\overline{\nu(z)}\mu(z)}\right|^2\rho_{\mathbb{U}}(H(z)) J_{H}(z) dxdy ,\\
        \end{split}
\end{equation*}
where $J_{H}$ denotes the Jacobian of $H$. We set $J=h^{-1}(I)$, where $h$ is the boundary extension of $H$ to $\mathbb R$.
By Lemma \ref{box}, there is a constant $\alpha \geq 1$ depending only on the bi-Lipschitz constant
$L=L(H)=L(H^{-1}) \geq 1$ such that $H^{-1}(Q_I) \subset Q_{\alpha J}$,
where $\alpha J$ is the interval with the same center as $J$, but with length $\alpha|J|$. Then, for some constant $K>0$
depending only on $\Vert \mu \Vert_\infty$ and $\Vert \nu \Vert_\infty$, we have that
\begin{equation*}
\mathcal{I} \leq K \iint_{Q_{\alpha J}} |\mu(z)-\nu(z)|^2\rho_{\mathbb{U}}(z)\frac{{\rm Im}\,z}{{\rm Im}\,H(z)}J_{H}(z) dxdy.
\end{equation*}

Since $H$ is bi-Lipschitz with respect to the hyperbolic metric, we see that
$$
\left(\frac{{\rm Im}\,z}{{\rm Im}\,H(z)}\right)^2J_{H}(z) \asymp 1. 
$$
Moreover, Lemma \ref{ratio} implies that
$$
\frac{{\rm Im}\,z}{{\rm Im}\,H(z)} \asymp \frac{|I_z|}{|h(I_z)|}.
$$
Both comparabilities as above are given by constants depending only on the bi-Lipschitz constant $L$. 
Hence, there is a constant $\widetilde L \geq 1$
depending only on $L$ such that
\begin{equation*}
\widetilde L^{-1}\frac{|h(I_z)|}{|I_z|} \leq \frac{{\rm Im}\,z}{{\rm Im}\,H(z)}J_{H}(z) \leq \widetilde L\, \frac{|h(I_z)|}{|I_z|} 
\end{equation*}
for every $z \in \mathbb U$. By setting $\widetilde \omega(z)=|h(I_z)|/|I_z|$, we have
\begin{equation*}
\mathcal{I} \leq K \widetilde L \iint_{Q_{\alpha J}} \widetilde \omega(z) |\mu(z)-\nu(z)|^2\rho_{\mathbb{U}}(z) dxdy. 
\end{equation*}

Now we apply the following {\it Carleson embedding theorem} (see \cite[Theorem 1.57]{Saw}). 
It says that assuming $\lambda \in {\rm CM}(\mathbb U)$ is a Carleson measure on the upper half-plane $\mathbb{U}$ and $F: \mathbb U \to [0, \infty)$ is a non-negative Borel measurable function in general, we have
\begin{equation*}
    \iint_{\mathbb{U}}F(z) d\lambda(z) \leq A \Vert \lambda \Vert_{c}\int_{\mathbb{R}}F^{\ast}(t) dt,
\end{equation*}
where $F^{\ast}(t) = \sup_{z \in \Gamma(t)}F(z)$ denotes the non-tangential maximal function of $F$ at $t \in \mathbb{R}$ with 
a cone $\Gamma(t) = \{z=(x,y) \in \mathbb{U}\mid |x-t| \leq y\}$, and $A>0$ is an absolute constant. 

By the assumption $\mu, \nu \in \mathcal M(\mathbb U)$, the measure $\lambda_{\mu-\nu}$ defined by
$$
d\lambda_{\mu-\nu}=|\mu(z)-\nu(z)|^2\rho_{\mathbb{U}}(z)dxdy
$$ 
belongs to $\rm CM(\mathbb{U})$.
Then, we obtain that
\begin{equation*}
\begin{split}
\mathcal{I} & \leq K \widetilde L \iint_{\mathbb{U}} \widetilde \omega(z)1_{Q_{\alpha J}}(z) 
|\mu(z)-\nu(z)|^2\rho_{\mathbb{U}}(z) dxdy  \\
& \leq AK \widetilde L\,\Vert \lambda_{\mu-\nu} 1_{Q_{\alpha J}} \Vert_c \int_{\mathbb{R}} (\widetilde \omega 1_{Q_{\alpha J}})^{\ast}(t) dt = C(J)\int_{3\alpha J} (\widetilde \omega 1_{Q_{\alpha J}})^{\ast}(t) dt,\\
\end{split}
\end{equation*}
where the constant $C(J)=AK \widetilde L\,\Vert \lambda_{\mu-\nu} 1_{Q_{\alpha J}} \Vert_c>0$ depends also on 
the interval $J$, but is bounded due to $\lambda_{\mu-\nu} \in \rm CM(\mathbb{U})$.

The boundary extension $h$ of $H$ to $\mathbb R$ is strongly quasisymmetric by Proposition \ref{SQSext}.
We consider the $A_\infty$-weight $\omega=h'$ on $\mathbb R$ and set
$\varphi = \omega 1_{3\alpha J}$. Let 
$$
M\varphi(t) = \sup_{t \in I}\frac{1}{|I|}\int_{I}\varphi(s) ds
$$
denote the Hardy--Littlewood (uncentered) {\it maximal function} of $\varphi$. 
Then, we can show that $(\widetilde \omega 1_{Q_{\alpha J}})^{\ast}(t)\leq M\varphi(t)$ 
for any $t\in 3 \alpha J$. Indeed, $t \in I_z \subset 3 \alpha J$ for every $z \in \Gamma(t) \cap Q_{\alpha J}$, and hence
\begin{align*}
(\widetilde \omega 1_{Q_{\alpha J}})^{\ast}(t) =\sup_{z \in \Gamma(t) \cap Q_{\alpha J}} \widetilde \omega(z)
&=\sup_{z \in \Gamma(t) \cap Q_{\alpha J}}\frac{|h(I_z)|}{|I_z|}\\
&\leq \sup_{t \in I_z \subset 3 \alpha J} \frac{1}{|I_z|}\int_{I_z} \omega(t)dt \leq M\varphi(t).
\end{align*}
Therefore, 
$$
\mathcal{I} \leq C(J)\int_{3\alpha J} M\varphi(t) dt.
$$

By the {\it reverse H\"older inequality} in the Muckenhoupt theory (see \cite[Lemma 5]{CF}), 
there exist $C>0$ and $p >1$ such that for any bounded closed
interval $I \subset \mathbb R$, we have
\begin{equation*}
    \frac{1}{|I|}\int_{I}\omega(t)^p dt \leq C\left(\frac{1}{|I|}\int_{I}\omega(t) dt\right)^p.
\end{equation*}
It follows that 
\begin{equation*}
    \begin{split}
        \mathcal{I} & \leq C(J)|3\alpha J|^{\frac{1}{p'}} \left(\int_{3\alpha J} (M\varphi)^p\right)^{\frac{1}{p}} 
   \leq C_1(J)|3 \alpha J|^{\frac{1}{p'}} \left(\int_{\mathbb{R}} \varphi^p\right)^{\frac{1}{p}} \\
         & = C_1(J)|3\alpha J|^{\frac{1}{p'}} \left(\int_{3\alpha J} \omega^p\right)^{\frac{1}{p}} 
         \leq C_2(J)\int_{3\alpha J}\omega = C_2(J)|h(3\alpha J)| \\
    \end{split}
\end{equation*}
for $1/p+1/p'=1$,
where $C_1(J)$ and $C_2(J)$ are some constant multiples of $C(J)$. We have used the H\"older inequality in the first inequality,
the strong $L^p$-estimate for maximal functions (see \cite[Theorem I]{CF}, \cite[Theorem 4.3]{Ga}) in the second inequality, and
the reverse H\"older inequality mentioned above 
in the last inequality.

Finally, we estimate $|h(3\alpha J)|$. Let $M \geq 1$ be the quasisymmetry constant of $h$.
Then, 
$$
|h(3\alpha J)| \leq \frac{M^{3\alpha}-1}{M-1}|h(J)|=\widetilde M|I|,
$$
where the constant involving $M$ and $\alpha$ is replaced with $\widetilde M \geq 1$.
This yields an estimate 
$$
\frac{\mathcal{I}}{|I|} =\frac{1}{|I|}\iint_{Q_I} |\mu_{G \circ H^{-1}}(\zeta)|^2 \rho_{\mathbb{U}}(\zeta)d\xi d\eta\leq \widetilde M C_2(J).
$$

Since $C_2(J)$ is bounded independent of $J$, we see that $\mu_{G \circ H^{-1}}$ belongs to $\mathcal M(\mathbb U)$.
This proves statement (1). We further assume that
$\mu, \nu \in \mathcal M_0(\mathbb U)$. Then, 
$\lambda_{\mu-\nu} \in \rm CM_0(\mathbb{U})$.
Since
$$
C_2(J) \asymp \Vert \lambda_{\mu-\nu} 1_{Q_{\alpha J}} \Vert_c,
$$
this tends uniformly to $0$ as $|J| \to 0$. 
Moreover, by the assumption that $h^{-1}$ is uniformly continuous on $\mathbb R$,
we see that $|J| \to 0$ uniformly as $|I| \to 0$. 
This shows that $\mu_{G \circ H^{-1}}$ belongs to $\mathcal M_0(\mathbb U)$, which proves statement (2).
\end{proof}

\section{The barycentric extension: Proof of Theorem \ref{DEonR}}

In this section, after introducing the barycentric extension,
we prove Theorem \ref{DEonR}, and provide a complex-variable proof for Theorem \ref{thm2}. 
We also show the fact that the group ${\rm SS}_{\rm uc}$ acts on the 
Teichm\"uller space $T_v(\mathbb R)$ as biholomorphic automorphisms.

We first recall the {\it barycentric extension} for a homeomorphism of the unit circle $\mathbb S$ 
introduced by Douady and Earle \cite{DE}, and then we translate it to the setting of the real line $\mathbb R$. 

For $\varphi \in {\rm QS}(\mathbb S)$,
the average of $\varphi$ taken at $w \in \mathbb D$ is defined by
$$
\xi_\varphi(w)=\frac{1}{2\pi} \int_{\mathbb S}\gamma_w(\varphi(\zeta))|d\zeta|
$$
where the M\"obius transformation 
$$
\gamma_w(z)=\frac{z-w}{1-\bar w z} \in \mbox{\rm{M\"ob}}(\mathbb D)
$$
maps $w$ to the origin $0$.
The barycenter of $\varphi$ is the unique point $w_0 \in \mathbb D$ such that $\xi_\varphi(w_0)=0$.
The value of the barycentric extension $e(\varphi)$ at the origin $0$ is defined to be the barycenter $w_0=w_0(\varphi)$, that is,
$e(\varphi)(0)=w_0(\varphi)$.
For an arbitrary point $z \in \mathbb D$, the barycentric extension $e(\varphi)$ is defined by
$$
{e(\varphi)(z)=e(\varphi \circ \gamma_z^{-1})(0)},
$$
which satisfies the {\it conformal naturality} such that
$$
{e(\gamma_1 \circ \varphi \circ \gamma_2)=e(\gamma_1) \circ e(\varphi) \circ e(\gamma_2)}
$$
for any $\gamma_1, \gamma_2 \in \mbox{\rm{M\"ob}}(\mathbb S)$. Moreover, $e(\varphi)$ is a quasiconformal 
diffeomorphism of $\mathbb D$ onto itself and is even bi-Lipschitz under the hyperbolic metric on $\mathbb D$. 

We fix a Cayley transformation $T:\mathbb{U} \to \mathbb{D}$ given by
$w = T(z) = (z - i)/(z + i)$. 
For a quasisymmetric homeomorphism $h \in {\rm QS}(\mathbb R)$, we
set $\varphi = T \circ h \circ T^{-1}$, and define $e(h) =T^{-1} \circ e(\varphi) \circ T$, where $e(\varphi)$ is the barycentric extension of $\varphi \in {\rm QS}(\mathbb S)$. 
We also call $e(h)$ the barycentric extension of $h$. The complex dilatation of $e(h)$
is denoted by $\mu_{e(h)}$. 
The barycentric extension $e(h)$ on $\mathbb U$ also satisfies the conformal naturality and bi-Lipschitz continuity 
as in the case of $\mathbb D$.

Suppose that $\varphi \in {\rm QS}(\mathbb{S})$ has a quasiconformal extension to $\mathbb{D}$ with complex dilatation $\mu$.  We denote by $e^{-1}(\varphi^{-1})$ the inverse mapping of the barycentric extension
$e(\varphi^{-1})$. By checking the proof of \cite[Corollary 3.4]{TWS}, we see that 
$\mu_{e^{-1}(\varphi^{-1})} \in \mathcal{M}(\mathbb{D})$ if  $\mu \in \mathcal{M}(\mathbb{D})$, and 
$\mu_{e^{-1}(\varphi^{-1})} \in \mathcal{M}_{0}(\mathbb{D})$ if $\mu \in \mathcal{M}_0(\mathbb{D})$. 
Furthermore, the results on the barycentric extension $e(\varphi)$ itself were deduced from this
by using \cite[Lemma 10]{AZ} or \cite[Lemme 4.8]{Se} (which are generalized to Theorem \ref{biLipschitz}). Namely, 
$\mu_{e(\varphi)} \in {\mathcal M}(\mathbb{D})$ if $\mu \in {\mathcal M}(\mathbb{D})$, and
$\mu_{e(\varphi)} \in {\mathcal M}_0(\mathbb{D})$ if $\mu \in {\mathcal M}_0(\mathbb{D})$.
This claim is also true for ${\mathcal M}(\mathbb{U})$ on the upper half-plane $\mathbb U$ by the conformal invariance, but is
not known to be true for ${\mathcal M}_0(\mathbb{U})$.
The former result on $\mu_{e^{-1}(\varphi^{-1})}$ was translated to the setting on $\mathbb U$ 
in \cite[Theorem 4.3]{Sh19} as follows.

\begin{lemma}\label{DEU}
Let $H$ be a quasiconformal homeomorphism of $\mathbb U$ onto itself whose complex dilatation 
is in ${\mathcal M}_0(\mathbb{U})$. Then, for the boundary extension $h$ of $H$ to $\mathbb R$,
the complex dilatation $\mu_{e^{-1}(h^{-1})}$ is also in ${\mathcal M}_0(\mathbb{U})$.
\end{lemma}

This lemma implies that once we obtain such a quasiconformal homeomorphism $H$, we can
replace it with the bi-Lipschitz diffeomorphism given by means of the barycentric extension
with the same boundary extension $h$ as $H$ and with its complex dilatation in the same class ${\mathcal M}_0(\mathbb{U})$ as $H$.
For a given $h \in {\rm SS}(\mathbb R)$, the existence of the appropriate quasiconformal extension $H$
is guaranteed by Proposition \ref{SSext}. Thus, we can prepare the following claim for the proof of Theorem \ref{DEonR}.

\begin{proposition}\label{step}
Any $h \in {\rm SS}(\mathbb R)$ extends continuously to a bi-Lipschitz diffeomorphism $e^{-1}(h^{-1})$ of $\mathbb U$ onto itself
whose complex dilatation belongs to ${\mathcal M}_0(\mathbb{U})$.
\end{proposition}

\begin{remark}
On the unit disk $\mathbb D$, the barycentric extension $e(\varphi)$ satisfies 
$\mu_{e(\varphi)} \in {\mathcal M}(\mathbb{D})$ for $\varphi \in {\rm SQS}(\mathbb S)$ and 
$\mu_{e(\varphi)} \in {\mathcal M}_0(\mathbb{D})$ for $\varphi \in {\rm SS}(\mathbb S)$. 
These facts follow from the above arguments combined with Proposition \ref{SQSext} (applied to $\mathbb D$) 
and Proposition \ref{SW4.1},
which originally appeared in \cite{CZ} and \cite{TWS}.
\end{remark}

\begin{proof}[Proof of Theorem \ref{DEonR}]
Since $h \in \rm SS(\mathbb R)$ and $h^{-1}$ is uniformly continuous on $\mathbb R$, we conclude by Corollary \ref{3item} that $h^{-1} \in \rm SS(\mathbb R)$. 
Then by Proposition \ref{step}, we have $\mu_{e^{-1}(h)} \in \mathcal{M}_0 (\mathbb U)$. Since 
$e^{-1}(h)$ is bi-Lipschitz on $\mathbb U$ and
$h$ is uniformly continuous on $\mathbb R$, Theorem \ref{biLipschitz} implies that $\mu_{e(h)} \in \mathcal{M}_0 (\mathbb U)$. This completes the proof. 
\end{proof}

The complex-variable proof of Theorem \ref{thm2} is also instructive. We give the argument by using 
the results on quasiconformal extensions in Theorem \ref{biLipschitz}, Proposition \ref{SSext}, and Proposition \ref{step}.

\begin{proof}[Another proof of Theorem \ref{thm2}]
For $g, h \in \rm SS(\mathbb R)$, we set $G=e^{-1}(g^{-1})$ and $H=e^{-1}(h^{-1})$ (only for $H$, this particular
construction is necessary). By Proposition \ref{step}, we have $\mu_G,\, \mu_H \in \mathcal{M}_0 (\mathbb U)$. 
Since $H=e^{-1}(h^{-1})$ is bi-Lipschitz with respect to the hyperbolic metric on 
$\mathbb U$ and $h^{-1}$ is uniformly continuous on $\mathbb R$ by assumption,  
we conclude by Theorem \ref{biLipschitz} that $\mu_{G \circ H^{-1}} \in \mathcal{M}_0 (\mathbb U)$.
Then, the boundary extension $g \circ h^{-1}$ of $G \circ H^{-1}$ belongs to ${\rm SS}(\mathbb R)$ by Proposition \ref{SSext}.
\end{proof}

Finally, we mention the action of the group ${\rm SS}_{\rm uc}(\mathbb R)$ on $T_v(\mathbb R)$. 
For each $h \in {\rm SS}_{\rm uc}(\mathbb R)$, this can be simply defined by
$h_*([g])=[g \circ h^{-1}]$ for $[g] \in T_v(\mathbb R)$, but
to see that this gives a biholomorphic automorphism of $T_v(\mathbb R)$, we have to extend $h$
quasiconformally to $\mathbb U$ and consider its action on ${\mathcal M}_0(\mathbb U)$.

\begin{theorem}\label{action}
For any $h \in {\rm SS}_{\rm uc}(\mathbb R)$, let $H=e(h)$.
For any $\mu \in {\mathcal M}_0(\mathbb U)$,
let $G$ be a quasiconformal homeomorphism of $\mathbb U$ onto itself whose complex dilatation is $\mu$.
Then, the correspondence $\mu \mapsto \mu_{G \circ H^{-1}}$ defines a biholomorphic automorphism
$r_H:{\mathcal M}_0(\mathbb U) \to {\mathcal M}_0(\mathbb U)$.
Moreover, this map descends down to a biholomorphic automorphism $R_h:T_v(\mathbb R) \to T_v(\mathbb R)$ 
that coincides with $h_*$ and satisfies $\pi \circ r_H=R_h \circ \pi$ for the Teichm\"uller projection 
$\pi:{\mathcal M}_0(\mathbb U) \to T_v(\mathbb R)$. 
\end{theorem}

Indeed, in the proof of Theorem \ref{biLipschitz}, we obtain that
$$
\Vert \lambda_{\mu_{G \circ H^{-1}}} \Vert_c \lesssim \Vert \lambda_{\mu_{G}-\mu_{H}} \Vert_c
\lesssim \Vert \lambda_{\mu_{G}} \Vert_c+\Vert \lambda_{\mu_{H}} \Vert_c,
$$
from which we see that $r_H$ is locally bounded. Then, the remaining arguments for showing holomorphy
are carried out in a standard way.
See \cite[Remark 5.1]{SW} and \cite[Proposition 3.1]{WM21}.
The properties of the inverse maps are clear by $(r_{H})^{-1}=r_{H^{-1}}$ and $(R_h)^{-1}=R_{h^{-1}}$.

\section{Comparisons of $\rm SS (\mathbb{R})$ and $\rm SS(\mathbb{S})$ }
The conformal invariance of strongly quasisymmetric homeomorphisms is well understood. 
However, this is not the case for strongly symmetric homeomorphisms.
The problem comes from the uniformity of vanishing quantities related to VMO and vanishing Carleson measures.
In this section, we will clarify the relationship between strongly symmetric homeomorphisms 
on $\mathbb R$ and those on $\mathbb S$ along with some observations which may be of independent interests. 

We switch the definition of Carleson measure to measuring
the intersection of a disk with $\mathbb U$ or $\mathbb D$ instead of measuring a Carleson box or sector. More precisely,
a positive Borel measure $\lambda$ on $\mathbb D$ (similarly on $\mathbb U$) is a Carleson measure if
$$
\sup_{\Delta(\xi,r)} \frac{\lambda(\Delta(\xi,r) \cap \mathbb D)}{r} <\infty, 
$$
where the supremum is taken over all closed disks $\Delta(\xi,r)$ with center $\xi \in \mathbb S$ and radius $r \in (0,2)$.
A vanishing Carleson measure $\lambda \in \rm CM_0(\mathbb{D})$ is defined by verifying 
a uniform vanishing limit of the above quantity as $r \to 0$. 
These definitions are equivalent to the previous ones.

We begin with considering the correspondence between $\rm CM_0(\mathbb{U})$ and $\rm CM_0(\mathbb{D})$ under the Cayley transformation
$T:\mathbb U \to \mathbb D$ defined by $T(z) = (z - i)/(z + i)$. The following is a basic fact.

\begin{lemma}\label{CM0DU}
If $\lambda \in \rm CM_0(\mathbb{D})$, then the pull-back measure $T^*\lambda$ on $\mathbb{U}$ satisfying
$d(T^*\lambda)= |T'|^{-1} d\lambda\circ T $ belongs to $\rm CM_0(\mathbb{U})$.
\end{lemma}

\begin{proof}
Let $\Delta(x,r)$ denote a closed disk with center $x \in \mathbb R$ and radius $r>0$. Then,
\begin{align*}
&\quad\ \frac{1}{r} \iint_{\Delta(x,r) \cap \mathbb U} d(T^*\lambda)(z)
 = \frac{1}{r} \iint_{\Delta(x,r) \cap \mathbb U} |T'(z)|^{-1} d\lambda\circ T(z)\\
&=\frac{1}{{\rm rad}(T(\Delta(x,r)))} \iint_{T(\Delta(x,r)) \cap \mathbb D} 
\frac{{\rm rad}(T(\Delta(x,r)))}{{\rm rad}(\Delta(x,r))}|(T^{-1})'(w)| d\lambda(w).
\end{align*}
Here, ${\rm rad}(T(\Delta(x,r)))|(T^{-1})'(w)|/{\rm rad}(\Delta(x,r))$ is uniformly bounded by
some constant $C>0$ on $T(\Delta(x,r))$ 
for all sufficiently small $r>0$.

Since $\lambda \in \rm CM_0(\mathbb{D})$, 
for every $\varepsilon>0$, there is some $\delta>0$ such that if ${\rm rad}(T(\Delta(x,r)))<2\delta$ then
$$
\frac{C}{{\rm rad}(T(\Delta(x,r)))}\iint_{T(\Delta(x,r)) \cap \mathbb D} d \lambda(w) <\varepsilon.
$$
Moreover, ${\rm rad}(T(\Delta(x,r))) \leq 2\, {\rm rad}(\Delta(x,r))$ by $|T'(x)| \leq 2$ for all $x \in \mathbb R$.
Hence, if $r<\delta$ then
$$
\frac{1}{r} \iint_{\Delta(x,r) \cap \mathbb U} d(T^*\lambda)(z) <\varepsilon
$$
for all $x \in \mathbb R$. Therefore, we have $T^*\lambda\in \rm CM_0(\mathbb{U})$.
\end{proof}

The following is the main result in this section. This also compares $T_v$ with $T_v(\mathbb R)$.

\begin{theorem}\label{SR}
If $\varphi \in \rm SS(\mathbb{S})$, then $h = T^{-1}\circ \varphi \circ T \in  \rm SS(\mathbb{R})$.
The converse does not necessarily hold. In other words, there exists $h \in \rm SS (\mathbb{R})$ such that 
$\varphi = T \circ h \circ T^{-1} \notin \rm SS(\mathbb{S})$.
\end{theorem}

\begin{remark}
We point out that there is another way of constructing a strongly symmetric homeomorphism on $\mathbb R$ from
that on $\mathbb S$. This is done by taking a lift against the universal covering projection $\mathbb R \to \mathbb S$
defined by $x \mapsto e^{ix}$. Namely, for each sense-preserving homeomorphism $\varphi$ 
of $\mathbb S$ onto itself, there exists a strictly increasing homeomorphism $\hat{h}$ of $\mathbb R$ onto itself
that satisfies $\varphi(e^{ix}) = e^{i\hat{h}(x)}$. Then,
$\hat{h}(x + 2\pi) - \hat{h}(x) \equiv 2\pi$ and $\hat{h}'(x) = |\varphi'(e^{ix})|$.
If $\varphi \in \rm SS(\mathbb S)$, we have by Partyka \cite[Lemma 2.2]{Pa} that $\hat h \in {\rm SS}(\mathbb R)$.
\end{remark}

\begin{proof}[Proof of Theorem \ref{SR}]
By Proposition \ref{SW4.1},
any $\varphi \in \rm SS(\mathbb{S})$ extends to a quasiconformal homeomorphism $\Phi$ 
of $\mathbb{D}$ onto itself whose complex dilatation $\mu_{\Phi}$ induces a vanishing Carleson measure 
$\lambda_{\mu_{\Phi}} \in {\rm CM}_0(\mathbb D)$.
Then, $H = T^{-1}\circ \Phi \circ T$ is a quasiconformal extension of $h$ to $\mathbb{U}$ 
whose complex dilatation $\mu_{H}$ induces a measure $\lambda_{\mu_{H}}$ on $\mathbb U$ such that 
\begin{equation*}
    \begin{split}
       d\lambda_{\mu_{H}}(z)  & = |\mu_{H}(z)|^2\rho_{\mathbb{U}}(z)dxdy \\
      & = |\mu_{\Phi}(T(z))|^2\rho_{\mathbb{D}}(T(z))|T'(z)| dxdy\\
      & = |T'(z)|^{-1} d\lambda_{\mu_{\Phi}}\circ T(z) =d(T^*\lambda_{\mu_{\Phi}})(z).\\
    \end{split}
\end{equation*}
It follows from Lemma \ref{CM0DU} that $\lambda_{\mu_{H}} \in \rm CM_0(\mathbb{U})$, and thus 
$h \in \rm SS(\mathbb{R})$ by Proposition \ref{SSext}.

In order to prove the second assertion, we first recall 
 two functions $h$ and $g$ constructed in \cite{WM19}
(the roles of $h$ and $g$ are exchanged here). 

The function $h$ is simply defined as follows:
$$
h(x) = \begin{cases}
(x+1)^2-1, &  x \geq 0\\
-(x-1)^2+1, & x \leq 0.
\end{cases}
$$
To construct the function $g$, we consider a function $g_1(x)=x^2/24$ on the interval $[1,12] \subset \mathbb{R}$.
We draw the graph of $y=g_1(x)$ on the $xy$-plane and its $\pi$-rotating copy around the point $O=(1,g_1(1))$.
The union of these two curves is denoted by $\mathcal G_1$. Its end points are $E=(12, g_1(12))$ and
the antipodal point $E'$ on the copy. 
We move $\mathcal G_1$ by parallel translation so that $E'$ coincides with the origin $(0,0)$ of the $xy$-plane. In the positive direction, we 
put each $\mathcal G_1$ from one to another so that $E'$ coincides with $E$.
The resulting curve that is a graph on $\{x \geq 0\}$ is denoted by $\mathcal G_+$.
We also make its $\pi$-rotating copy around the origin $(0,0)$, which is denoted by $\mathcal G_-$. Then, we set 
$\mathcal G=\mathcal G_+ \cup \mathcal G_-$.
This curve $\mathcal G$ on the $xy$-plane defines a function $y=g(x)$ for $x \in \mathbb{R}$ that has $\mathcal G$ as its graph.

We have shown in \cite[Corollary 5.6]{WM19} that 
$g, h \in \rm SS(\mathbb R)$ but 
$g \circ h \notin \rm SS(\mathbb{R})$. 
Suppose that both $T \circ g \circ T^{-1}$ and $T \circ h \circ T^{-1}$ are in 
$\rm SS(\mathbb{S})$. Since $\rm SS(\mathbb{S})$ is a group,
we have that $T \circ g \circ h \circ T^{-1}$ is in $\rm SS(\mathbb{S})$. Then,  
we conclude by the above argument that $g \circ h \in \rm SS(\mathbb{R})$. 
This is a contradiction, and thus either $T \circ g \circ T^{-1}$ or $T \circ h \circ T^{-1}$ is not in $\rm SS(\mathbb{S})$. 
\end{proof}

\begin{remark}
{\rm
As an immediate consequence of Theorem \ref{SR}, we see that
there exists $\tilde \lambda \in \rm CM_0(\mathbb{U})$ such that the push-forward measure $T_*{\tilde \lambda}$ on $\mathbb D$ satisfying
$d(T_*{\tilde \lambda})=d((T^{-1})^*{\tilde \lambda}) =  |(T^{-1})'|^{-1} d\tilde \lambda \circ T^{-1}$ is not in $\rm CM_0(\mathbb{D})$.
}
\end{remark}

Next, we consider the correspondence between VMO functions on $\mathbb R$ and on $\mathbb S$.
Similar results to Theorem \ref{SR} can be obtained;
the boundedness of $|T'(x)|$ also transforms $\rm VMO(\mathbb{S})$ into $\rm VMO(\mathbb{R})$ under the Cayley transformation $T$.

\begin{proposition}\label{VMO(S)}
If $v \in \rm VMO(\mathbb{S})$, then $u = v \circ T \in \rm VMO(\mathbb{R})$.
The converse does not necessarily hold. In other words, there exists
$u \in \rm VMO(\mathbb{R})$ such that $v = u \circ T^{-1} \notin \rm VMO(\mathbb{S})$.
\end{proposition}

\begin{proof}
It is easy to see that
\begin{equation*}
    \frac{1}{|I|}\int_{I}|u(x) - u_I| dx \leq \frac{2}{|I|}\int_{I}|u(x) - c| dx 
\end{equation*}
for any bounded closed interval $I \subset \mathbb{R}$ and any $c \in \mathbb R$.
For $c = v_J$, $J = T(I)$ 
and $x = T^{-1}(\xi)$, the right side term in the above inequality turns out to be
\begin{equation*}
    \frac{2}{|J|}\int_{J}|v(\xi) - v_J| \frac{|J|}{|I|}|(T^{-1})'(\xi)| |d\xi|.
\end{equation*}
Here, $|J||(T^{-1})'(\xi)|/|I|$ is bounded from above by an absolute constant $C>0$ 
for all sufficiently small intervals $I$. 
We see that $|J|\to 0$ as $|I|\to 0$ by $|T'(z)|\leq 2$. Hence, 
\begin{equation*}
    \frac{1}{|I|}\int_{I}|u(x) -u_I| dx \leq \frac{2C}{|J|}\int_{J}|v(\xi) - v_J|  |d\xi| \to 0
\end{equation*}
as $|I| \to 0$ by the condition $v \in \rm VMO(\mathbb{S})$. This implies that $u \in \rm VMO(\mathbb{R})$.

For the converse direction, 
we set $v(\xi)=\log|1-\xi|$. This does not belong to $\rm VMO(\mathbb{S})$.
However, 
$$
u(x)=v \circ T(x) = -\log|x+i|+\log 2
$$ 
belongs to $\rm VMO(\mathbb{R})$ because it is in $\rm BMO(\mathbb{R})$ and is uniformly continuous on $\mathbb R$ (see \cite{Sa}).
\end{proof}

\begin{ackn}
The authors would like to thank the referee for a very careful reading of the manuscript and for several suggestions that greatly improve the presentation of the paper. 
\end{ackn}

\end{document}